\documentclass[final]{siamart171218}
\usepackage{amssymb,nicefrac,bm,upgreek,mathtools}
\usepackage[mathscr]{eucal}
\usepackage{dsfont}

\usepackage[normalem]{ulem}
\newcommand{\stkout}[1]{\ifmmode\text{\sout{\ensuremath{#1}}}\else\sout{#1}\fi}
\usepackage{enumitem}
\makeatletter
\def\namedlabel#1#2{\begingroup
    #2%
    \def\@currentlabel{#2}%
    \phantomsection\label{#1}\endgroup}
\makeatother

\newlist{describe}{description}{1}
\setlist[describe,1]{%
  font=\normalfont\textbf,
  itemindent=0pt,
  wide,
  itemsep=0pt,topsep=2pt}
\newsiamremark{remark}{Remark}
\newsiamremark{assumption}{Assumption}
\newsiamremark{notation}{Notation}
\newsiamremark{example}{Example}

\crefname{assumption}{Assumption}{Assumptions}
\newcommand{\df}{\coloneqq}
\DeclareMathOperator{\Exp}{\mathbb{E}} 
\DeclareMathOperator{\Prob}{\mathbb{P}} 
\newcommand{\D}{\mathrm{d}}          
\newcommand{\E}{\mathrm{e}}          
\newcommand{\RR}{\mathbb{R}}         
\newcommand{\Rd}{{\mathbb{R}^d}}       
\newcommand{\NN}{\mathbb{N}}         
\newcommand{\Ind}{\mathds{1}}            

\newcommand{\Act}{\mathbb{U}}        
\newcommand{\Uadm}{\mathfrak{U}}     

\newcommand{\Sob}{\mathscr{W}}       
\newcommand{\Sobl}{\mathscr{W}_{\mathrm{loc}}}  
\newcommand{\Cc}{C}                  

\newcommand{\lamstr}{\lambda^{\!*}}  
\newcommand{\ulamstr}{\underline{\lambda}^{\!*}} 
\newcommand{\olamstr}{\Bar{\lambda}^{\!*}} 

\newcommand{\transp}{^{\mathsf{T}}}  
\newcommand{\order}{{\mathscr{O}}}   
\newcommand{\sorder}{{\mathfrak{o}}} 

\newcommand{\Lg}{\mathscr{L}}        
\newcommand{\cL}{\mathcal{L}}        

\newcommand{\uuptau}{\Breve{\uptau}}
\newcommand{\grad}{\nabla}

\newcommand{\cA}{\mathcal{A}}
\newcommand{\sE}{\mathscr{E}}     
\newcommand{\sF}{\mathfrak{F}}    
\newcommand{\cG}{\mathcal{G}}    
\newcommand{\cH}{\mathcal{H}}
\newcommand{\cK}{\mathcal{K}}
\newcommand{\cP}{\mathcal{P}}     
\newcommand{\Lyap}{\mathscr{V}}   

\newcommand{\abs}[1]{\lvert#1\rvert}
\newcommand{\norm}[1]{\lVert#1\rVert}

\newcommand{\TheTitle}{A variational formula for risk-sensitive control}
\newcommand{\TheAuthors}{Ari Arapostathis and Anup Biswas}
\headers{\TheTitle}{\TheAuthors}

\title{{A variational formula for risk-sensitive\\ control of diffusions in $\Rd$}}

\author{Ari Arapostathis\thanks{Department of Electrical and Computer
Engineering, The University of Texas at Austin, 2501 Speedway, EER 7.824,
Austin, TX 78712 (\email{ari@ece.utexas.edu}).}
\and
Anup Biswas\thanks{Department of Mathematics,
Indian Institute of Science Education and Research,
Dr.\ Homi Bhabha Road, Pune 411008, India
(\email{anup@iiserpune.ac.in}).}}

\begin{document}
\maketitle

\begin{abstract}
We address the variational problem for the generalized principal
eigenvalue on $\Rd$ of linear and semi-linear  elliptic
operators associated with nondegenerate diffusions controlled through the drift.
We establish the Collatz--Wielandt formula for potentials
that vanish at infinity under minimal hypotheses, and also for general potentials
under blanket geometric ergodicity assumptions.
We also present associated results having the flavor of a
refined maximum principle.
\end{abstract}

\begin{keywords}
Principal eigenvalue; (Agmon) ground state; semilinear PDE;
Donsker--Varadhan functional; Collatz--Wielandt formula; min-max formula
\end{keywords}

\begin{AMS}
60J60, Secondary 60J25, 35P15, 60F10, 49G05
\end{AMS}

\section{Introduction}
Since the seminal work of Donsker and Varadhan \cite{DoVa-72,DoVa-76b},
a lot of effort has been devoted to variational characterizations
of principal eigenvalues of elliptic operators.
More recently, the work of
Berestycki, Nirenberg, and Varadhan \cite{Berestycki-94} opened
up the study of generalized eigenvalues in unbounded domains
(see also \cite{Berestycki-15}),
while advances in nonlinear Perron--Frobenius theory \cite{Lemmens-12,NP}
made possible the extension of the classical Collatz--Wielandt formula
for the Perron--Frobenius eigenvalue of irreducible non-negative matrices
to more abstract settings.
See also \cite[Chapter~3]{Pinsky} for
a Collatz--Wielandt formula
for symmetric second-order operators in bounded domains.

The motivation for this work is
the infinite horizon risk-sensitive control problem on the entire domain, which
seeks to minimize the asymptotic growth rate of the expected
`exponential of integral' cost, and which, under suitable assumptions,
coincides with the generalized principal eigenvalue of
the associated semilinear elliptic operator (for some recent
results see \cite{ari-anup,ABS}).
Recall the celebrated formula of Donsker--Varadhan:
for a uniformly elliptic nondivergence form operator $\Lg$ on a smooth bounded
domain $D\subset\Rd$, the principal eigenvalue $\lambda_1(\Lg,D)$
can be expressed as
\begin{equation*}
\lambda_1(\Lg,D)\,=\, \sup_{\varphi\in\Cc^{2,+}(D)}\,\inf_{\mu\in\cP(\Bar{D})}\,
\int_{D}\frac{\Lg \varphi(x)}{\varphi(x)}\,\mu(\D{x})\,,
\end{equation*}
where $\cP(\Bar{D})$ denotes the set of Borel probability measures
on $\Bar{D}$, and $\Cc^{2,+}(D)$ the space of positive functions in
$\Cc^2(D)\cap\Cc(\Bar{D})$.
Taking the supremum over measures,  followed by
the infimum over the function space, also results in equality,
and this forms an extension of the classical Collatz--Wielandt formula.
For versions of this formula for nonlinear operators on a bounded
domain see \cite{AS09,Quaas-08a}.

The Collatz--Wielandt formula for a reflected controlled diffusion
on a bounded domain has been studied in \cite{ABK16}
with the aid of nonlinear versions of the Krein--Rutman theorem.
Establishing this min-max formula over $\Rd$ is quite challenging, not only due
to the lack of compactness, but also because the generalized principal eigenvalue
of an operator does not enjoy all the structural properties of eigenvalues
over bounded domains. We take a different approach
which is based on the stochastic representation of
principal eigenfunctions to obtain several variational formulations of the
principal eigenvalues. 
For potentials that vanish at infinity, we exhibit the Collatz--Wielandt
formula under minimal assumptions (see \cref{T2.2}).
For more general potentials, we impose blanket geometric ergodicity assumptions
to handle the lack of compactness (see \cref{A2.1,A2.2}),
and establish the formula in \cref{T2.4}.
We then continue with two results in the flavor of a refined maximum
principle (see \cref{T2.5,T2.6}), and conclude the study
with some characterizations of the generalized principal eigenvalue
(\cref{T2.7,T2.8,T2.19}).
The proofs of these results are in \cref{S3}.

We would also like to mention
the recent work in \cite{AABK} which studies the  maximization
of the risk-sensitive average reward on the whole space,
without employing any blanket ergodicity assumptions.
The approach in \cite{AABK} leads to a concave maximization problem, which is
quite different from the `sup--inf' and `inf--sup' formulas in the current paper.

\section{Assumptions and main results}
\subsection{The controlled diffusion model}

Consider a controlled diffusion process $X = \{X_{t},\,t\ge0\}$
which takes values in the $d$-dimensional Euclidean space $\RR^{d}$, and
is governed by the It\^o  equation
\begin{equation}\label{E-sde}
\D{X}_{t} \;=\,b(X_{t},U_{t})\,\D{t} + \upsigma(X_{t})\,\D{W}_{t}\,.
\end{equation}
All random processes in \cref{E-sde} live in a complete
probability space $(\Omega,\sF,\Prob)$.
The process $W$ is a $d$-dimensional standard Wiener process independent
of the initial condition $X_{0}$.
The control process $U$ takes values in a compact, metrizable set $\Act$, and
$U_{t}(\omega)$ is jointly measurable in
$(t,\omega)\in[0,\infty)\times\Omega$.
The set $\Uadm$ of \emph{admissible controls} consists of the
control processes $U$ that are \emph{non-anticipative}:
for $s < t$, $W_{t} - W_{s}$ is independent of
\begin{equation*}
\sF_{s} \,\df\,\text{the completion of~}
\cap_{y>s}\sigma\{X_{0},U_{r},W_{r},\;r\le y\}
\text{~relative to~}(\sF,\Prob)\,.
\end{equation*}

We impose the following standard assumptions on the drift $b$
and the diffusion matrix $\upsigma$ to guarantee existence
and uniqueness of solutions.
\begin{itemize}
\item[(A1)]
\emph{Local Lipschitz continuity:\/}
The functions
$b\colon\RR^{d}\times\Act\to\RR^{d}$ and 
$\upsigma\colon\RR^{d}\to\RR^{d\times d}$
are continuous, and satisfy
\begin{equation*}
\abs{b(x,u)-b(y, u)} + \norm{\upsigma(x) - \upsigma(y)}
\,\le\,C_{R}\,\abs{x-y}\qquad\forall\,x,y\in B_R\,,\ \forall\, u\in\Act\, .
\end{equation*}
for some constant $C_{R}>0$ depending on $R>0$.

\item[(A2)]
\emph{Affine growth condition:\/}
For some $C_0>0$, we have
\begin{equation*}
\sup_{u\in\Act}\; \langle b(x,u),x\rangle^{+} + \norm{\upsigma(x)}^{2}\,\le\,C_0
\bigl(1 + \abs{x}^{2}\bigr) \qquad \forall\, x\in\RR^{d}\,.
\end{equation*}
\item[(A3)]
\emph{Nondegeneracy:\/} For each $R>0$, it holds that
\begin{equation*}
\sum_{i, j=1}^d a^{ij}(x) \xi_i \xi_j \ge C^{-1}_R \abs{\xi}^2 \quad \forall x\in B_R\,,
\end{equation*}
and for all $\xi=(\xi_1, \ldots, \xi_d)\transp\in\Rd$,
where  $a=\frac{1}{2}\upsigma\upsigma\transp.$
\end{itemize}

It is well known that under (A1)--(A3), for any admissible control
there exists a unique solution of \cref{E-sde}
\cite[Theorem~2.2.4]{book}.
We define the family of operators $\cL_u\colon\Cc^{2}(\RR^{d})\to\Cc(\RR^{d})$,
where $u\in\Act$ plays the role of a parameter, by
\begin{equation*}
\cL_u f(x) \,\df\,  a^{ij}(x)\,\partial_{ij} f(x)
+ b^{i}(x,u)\, \partial_{i} f(x)\,,\quad u\in\Act\,,\ x\in\Rd\,.
\end{equation*}
Here we adopt the notation
$\partial_{i}\df\tfrac{\partial~}{\partial{x}_{i}}$ and
$\partial_{ij}\df\tfrac{\partial^{2}~}{\partial{x}_{i}\partial{x}_{j}}$
for $i,j\in\{1,\dotsc,d\}$, and we
often use the standard summation rule that
repeated subscripts and superscripts are summed from $1$ through $d$.
Let $c(x,u)$ be a function in $\Cc(\Rd\times\Act, \RR)$ that is locally
Lipschitz in $x$ uniformly with respect to $u\in\Act$,
 and is bounded below in $\Rd$.
We consider the following semilinear operator
\begin{equation}\label{E-cG}
\cG f(x) \,\df\, a^{ij}(x)\,\partial_{ij} f(x)
+ \min_{u\in\Act}\, \bigl[b^{i}(x,u)\, \partial_{i} f(x)
+ c(x,u) f(x)\bigr]\,,\quad x\in\Rd\,.
\end{equation}
We remark that
as far as the results of the paper are concerned, local Lipschitz continuity
of $x\mapsto c(x,u)$ may be relaxed to local H\"older continuity.

\subsection{Statements of the main results.}
Let $D$ be a smooth bounded domain. Without any loss of generality we assume 
that $0\in D$.
The principal eigenvalue of $\cG$ with Dirichlet boundary condition
is defined as follows:
\begin{equation}\label{E-lamD}
\begin{aligned}
\lambda_D(\cG) &\,\df\,\inf\,\bigl\{\lambda\,\colon \exists\,
\psi\in\Cc(\bar D)\cap\Cc^2(D),\; \psi>0
\text{\ in\ } D,\, \\
&\mspace{250mu}\text{\ satisfying\ }
\cG\psi -\lambda\psi\le 0 \text{\ in\ } D \bigr\}\,.
\end{aligned}
\end{equation}
It is then known from \cite[Theorem~1.1]{Quaas-08a} that there exists
a unique $\Psi=\Psi_D\in\Cc(\bar D)\cap\Cc^2(D)$ with $\Psi(0)=1$, $\Psi>0$ in $D$,
which satisfies
\begin{equation}\label{E-Psi}
\cG\Psi\,=\, \lambda_D(\cG)\Psi\quad \text{in\ } D\,,
\quad \Psi=0\quad \text{on\ } \partial D\,.
\end{equation}
By $\Cc^{2,+}(D)$ we denote the set of functions in $\Cc^2(D)\cap \Cc(\bar D)$
that are positive in $D$, and $\Cc^{2, +}_0(D)$ denotes the collection of functions
in $\Cc^{2,+}(D)$ that vanish on $\partial D$.
Our first result establishes a Collatz--Wielandt formula for $\lambda_D$.
The representation \cref{ET2.1B} below can also be found in \cite{AS09}, where
it plays a crucial role in obtaining necessary and sufficient conditions
for the solvability of certain Dirichlet problems.

\begin{theorem}\label{T2.1}
Let $D\subset\Rd$ be a smooth bounded domain.
Then
\begin{align}
\lambda_D(\cG) &\,=\, \sup_{\psi\in\Cc^{2,+}_0(D)}\; \inf_{\mu\in\cP(D)}\;
\int_{D}\frac{\cG\psi}{\psi}\, \D{\mu}\label{ET2.1A} \\
&\,=\, \inf_{\psi\in\Cc^{2,+}(D)}\; \sup_{\mu\in\cP(D)}\;
\int_{D}\frac{\cG\psi}{\psi}\, \D{\mu}\,,\label{ET2.1B}
\end{align}
where $\cP(A)$ denotes the set of all Borel probability measures on the set $A$.
\end{theorem} 

\begin{remark}
The function space $\Cc^{2,+}_0(D)$ in the representation formula \eqref{ET2.1A}
cannot, in general, be enlarged to $\Cc^{2, +}(D)$.
To see this consider any smooth domain $D_1$
strictly containing $D$. Let $\lambda_1=\lambda_{D_1}$ be the Dirichlet
principal eigenvalue of $\cG$ in $D_1$, and $\Psi_1$
denote the corresponding (positive) principal eigenfunction.
It is known that $\lambda_1>\lambda_D(\cG)$ \cite[Remark~3]{Quaas-08a}.
Take $\psi=\Psi+\Psi_1$. 
Then by the concavity of $\cG$, we have
\begin{equation*}
\inf_{\mu\in\cP(D)}\;\int_{D}\frac{\cG\psi}{\psi}\,\ge\,\min_{\bar D}\,
\frac{\lambda_1\Psi_1+\lambda_D(\cG)\Psi}{\Psi_1+\Psi}>\lambda_D\,.
\end{equation*}
\end{remark}

Our next goal is to establish a similar characterization for the generalized
principal eigenvalue of $\cG$ in $\Rd$.
To begin with, we consider the uncontrolled problem.
In this case, we have a linear operator of the form
\begin{equation}\label{E-Lg}
\Lg f(x)\,\df\,  a^{ij}(x)\,\partial_{ij} f(x)
+ b^{i}(x)\, \partial_{i} f(x) + c(x) f(x)\quad \text{in}\ \Rd\,.
\end{equation}
Here, we assume that $b, c$ are locally bounded, Borel measurable functions,
and that
$a$ is continuous and satisfies (A3).
We recall the definition of  the principal eigenvalue of
$\Lg$ from \cite{Berestycki-15},
denoted as $\lamstr(\Lg)$.
\begin{equation}\label{E-lamL}
\begin{aligned}
\lamstr(\Lg) &\,\df\,\inf\,
\bigl\{\lambda\in\RR\,\colon \exists\, \psi\in\Sobl^{2, d}(\Rd),\;
\psi>0,  \\
&\mspace{250mu}\text{\ satisfying\ } \Lg\psi-\lambda\psi\le 0
\text{\ a.e.\ in\ } \Rd\bigr\}\,.
\end{aligned}
\end{equation}
Note the analogy between \cref{E-lamD} and \cref{E-lamL}.

We start by showing that if
$\Lg$ has smooth coefficients, and
$\lamstr(\Lg)<\infty$, then
\begin{equation}\label{E-general}
\lamstr(\Lg)\,=\,
\inf_{\psi\in\Cc^{2,+}(\Rd)}\;
\sup_{\mu\in\cP(\Rd)}\;\int_{\Rd}\frac{\Lg\psi}{\psi}\, \D{\mu}\,.
\end{equation}
This is essentially in (1.12)--(1.13) of \cite{Berestycki-94}.
We can prove this from the definition of $\lamstr(\Lg)$ and the existence
of an eigenfunction, or can use the following argument.
If not, then there exists $\psi\in\Cc^{2,+}(\Rd)$ and $\epsilon>0$
such that
$\Lg\psi< (\lamstr(\Lg)-\epsilon)\psi$ on $\Rd$.
Let $\lambda_n$ denote the principal eigenvalue of $\Lg$ in $B_n$,
and choose $n$ large enough so that $\lambda_n> \lamstr(\Lg)-\epsilon$.
With $\psi_n$ denoting the principal eigenfunction on $B_n$ we have
\begin{equation*}
\Lg(\psi-\psi_n) - \lambda_n(\psi-\psi_n)
\,\le\, (\lamstr(\Lg)-\epsilon-\lambda_n)\psi\,.
\end{equation*}
Scaling $\psi$ so that it touches $\psi_n$ at some point from above,
and applying the strong maximum principle, we obtain
$\psi=\psi_n$ on $B_n$, which is not possible since $\psi_n$ vanishes
on $\partial B_n$.
The analogous result holds for the semilinear operator $\cG$.

We next show that the Collatz--Wielandt formula in \cref{ET2.1A}
does not hold, in general, for $\lamstr(\Lg)$.
Consider the generalized eigenvalues $\lambda'(\Lg)$ and $\lambda''(\Lg)$
defined  by
\begin{align*}
\lambda'(\Lg)&\,\df\,\sup\,\bigl\{\lambda\in\RR\,\colon\exists\,
\psi\in\Sobl^{2,d}(\Rd)\cap L^\infty(\Rd),
\; \psi>0, \\
&\mspace{320mu}\text{\ satisfying\ } \Lg\psi-\lambda\psi\ge 0 \text{\ a.e.\ in\ }
\Rd\bigr\}\,,\\[5pt]
\lambda''(\Lg)&\,\df\,\inf\,\Bigl\{\lambda\in\RR\,\colon\exists\,
\psi\in\Sobl^{2,d}(\Rd),
\; \inf_{\Rd}\,\psi>0, \\
&\mspace{320mu}\text{\ satisfying\ } \Lg\psi-\lambda\psi\le 0 \text{\ a.e.\ in\ }
\Rd\Bigr\}\,.
\end{align*}
It is known that, in general, $\lamstr(\Lg)\le \lambda'(\Lg)$
(see \cite[Theorem~1.7]{Berestycki-15}).
But this inequality might be strict \cite{BR06}.

\begin{example}\label{Eg2.1}
We borrow this example from \cite{BR06}.
Consider the operator $\Lg\phi\df\phi^{\prime\prime}-\phi^\prime$,
with $d=1$.
If $\psi\in L^\infty(\RR)$ satisfies
$\Lg\psi-\lambda\psi\ge 0$, then
applying the It\^{o}--Krylov formula
we obtain $\psi(x) \le \E^{-\lambda t} \norm{\psi}_\infty$ for all $t\ge0$.
Taking logarithms, it follows that $\lambda\le0$.
On the other hand, for $\psi=1$ we have $\Lg\psi=0$ and therefore,
we obtain
\begin{equation*}
\sup_{\psi\in\Cc^{2,+}_b(\RR)}\; \inf_{\mu\in\cP(\RR)}\,
\int_{\RR}\frac{\Lg \psi}{\psi}\,\D{\mu}
\,\ge\, 0\,=\,\lambda'(\Lg)\,.
\end{equation*}
For $R>0$, with $\phi_R(x)=\cos(\frac{\pi}{2R}x)\exp(\frac{x}{2})$,
we have
\begin{equation*}
\Lg\phi_R \,=\,
-\biggl(\frac{1}{4}+\frac{\pi^2}{4R^2}\biggr)\phi_R\quad \text{in\ } [-R, R]\,.
\end{equation*}
Using \cite[Proposition~3.1]{BR06} we deduce that
$\lamstr(\Lg)=\lim_{R\to\infty}-(\frac{1}{4}+\frac{\pi^2}{4R^2})=-\frac{1}{4}$.
Thus we obtain
\begin{equation*}
\lamstr(\Lg)\,<\, \sup_{\psi\in\Cc^{2,+}_b(\RR)}\; \inf_{\mu\in\cP(\RR)}\,
\int_{\RR}\frac{\Lg \psi}{\psi}\, \D{\mu}\,.\end{equation*}
\end{example}

In analogy to \cref{E-lamL} we define the principal eigenvalue on $\Rd$ of
the semilinear operator $\cG$ as follows
\begin{equation}\label{E-lamstr}
\begin{aligned}
\lamstr(\cG) &\,\df\,
\inf\,\bigl\{\lambda\in\RR\,\colon \exists\, \psi\in\Cc^2(\Rd),\; \psi>0,\\
&\mspace{250mu}
\text{\ satisfying\ }  \cG\psi-\lambda\psi\le 0 \text{\ a.e.\ in\ } \Rd\bigr\}\,.
\end{aligned}
\end{equation}
As in the case of the linear operator, we have the following characterization
of the
principal eigenvalue for the semilinear operator.

\begin{lemma}\label{L2.1}
Let $\lambda_n$ be the principal eigenvalue of $\cG$ in $B_n$ i.e.,
for some positive $\Psi_n\in\Cc^2(B_n)\cap\Cc(\bar{B}_n)$ we have
\begin{equation*}
\cG\Psi_n\,=\,\lambda_n\Psi_n\quad \text{in}\ B_n\,,
\quad \Psi_n=0\quad \text{on}\; \partial B_n\,.
\end{equation*}
Then $\lim_{n\to\infty} \lambda_n=\lamstr(\cG)$.
\end{lemma}

\begin{proof}
In view of \cref{E-lamD}, we note that $\lambda_n\le \lambda_{n+1}$ for all $n$.
By the definition in \cref{E-lamstr} we have $\lambda_n\le \lamstr(\cG)$.
Thus, $\lim_{n\to\infty}\lambda_n=\Hat\lambda\le \lamstr(\cG)$.
Using a standard argument of elliptic PDE,
we can find a positive $\Hat\Phi\in\Cc^2(\Rd)$ satisfying
\begin{equation*}
\cG\Hat\Phi \,=\, \Hat\lambda\,\Hat\Phi \quad \text{in}\ \Rd\,.
\end{equation*}
See for instance \cite{ari-anup,biswas-11a}.
By \cref{E-lamstr}, we then have $\lamstr(\cG)\le\Hat\lambda$.
Therefore, $\lim_{n\to\infty}\lambda_n=\lamstr(\cG)$.
\end{proof}

Next, consider the extremal operator $\cH$ defined by
\begin{equation}\label{E-Lh}
\cH f(x) \,\df\, a^{ij}(x)\,\partial_{ij} f(x)
+ \max_{u\in\Act}\,\bigl[b^{i}(x,u)\, \partial_{i} f(x) + c(x,u) f(x)\bigr]\,.
\end{equation}
This operator corresponds to the maximization problem of
the risk-sensitive ergodic average.
The principal eigenvalue $\lambda_D(\cH)$ is defined
in the same fashion as in \eqref{E-lamD}.
We show that \cref{ET2.1B} holds for the operator $\cH$ with the `inf' and `sup'
in reverse order.
This result is also used in \cref{T2.2} below.

\begin{theorem}\label{T2.18}
Let $\lambda_D(\cH)$
be the Dirichlet principal eigenvalue of $\cH$, where $D$ is
a bounded smooth domain or $\Rd$.
Then we have
\begin{equation}\label{ET2.18A}
\sup_{\mu\in\cP(D)}\; \inf_{\psi\in\Cc^{2,+}(D)}\; 
\int_{D}\frac{\cH\psi}{\psi}\, \D{\mu}\,=\,\lambda_D(\cH).
\end{equation}
\end{theorem}

We return to the operators $\Lg$ and $\cG$ to state the main results.
Note that the process associated the operator $\Lg$ in \cref{Eg2.1} is transient.
Our first result establishes a Collatz--Wielandt formula for
$\lamstr(\Lg)$,
when the underlying process is recurrent, and $c$ is bounded.
We let $\Cc^{2,+}_b(\Rd)\df \Cc_b^+(\Rd)\cap\Cc^{2}(\Rd)$,
where $\Cc_b^+(\Rd)$ denotes the set of positive bounded functions on $\Rd$.
Also, $\Cc^{2,+}(\Rd)$ denotes the class of positive functions
in $\Cc^{2}(\Rd)$.

\begin{theorem}\label{T2.2}
Consider the linear operator $\Lg$ in \cref{E-Lg}, and assume that
$b$ and $c$ are locally H\"older continuous, and 
$c$ is a function that vanishes at infinity.
Suppose that the process $X$ is recurrent.
Then, if $\lamstr(\Lg) > 0$, we have
\begin{equation}\label{ET2.2A}
\begin{aligned}
\lamstr(\Lg) &\,=\, \sup_{\psi\in\Cc^{2,+}_b(\Rd)}\;
\inf_{\nu\in\cP(\Rd)}\;\int_{\Rd}\frac{\Lg\psi}{\psi}\, \D{\nu}\\
&\,=\, \inf_{\psi\in\Cc^{2,+}_b(\Rd)}\;
\sup_{\nu\in\cP(\Rd)}\; \int_{\Rd}\frac{\Lg\psi}{\psi}\, \D{\nu}\\
&\,=\,
\sup_{\nu\in\cP(\Rd)}\;\inf_{\psi\in\Cc^{2,+}_b(\Rd)}\;
\int_{\Rd}\frac{\Lg\psi}{\psi}\, \D{\nu} \,.
\end{aligned}
\end{equation}
In general, i.e., independent on the sign of $\lamstr(\Lg)$,
\cref{ET2.2A} holds if we replace
$\Cc^{2,+}_b(\Rd)$ with $\Cc^{2,+}(\Rd)$ in the  second and third equalities.
Moreover, the first equality also holds for $\lamstr(\Lg)\le0$.
The analogous result for the first two equalities holds for the
semilinear operator $\cG$ in \cref{E-cG},
provided that $\sup_{(x,u)\in B_n^c\times\Act}\, c(x,u) \to 0$
as $n\to\infty$,
and under the assumption that the process $X$ is recurrent under any
stationary Markov control.
\end{theorem}

Given a set $A$, the first exit time from $A$ is denoted by
\begin{equation*}
\uptau(A) \,=\, \inf\{t>0\,\colon X_t\notin A\}\,.
\end{equation*}
For the first hitting time to the ball $B_r$ we use the abbreviated
notation $\uuptau_r=\uptau(B^c_r)$.
We also let $\uptau_{r}\df \uptau(B_{r})$.

\begin{remark}
Suppose $\lamstr(\Lg)<0$,
$c\in\Cc_0(\Rd)$, $a$ and $b$ are bounded, and
the diffusion is geometrically ergodic.
Then there is no $\psi\in \Cc^{2, +}_b(\Rd)$ satisfying
$$\sup_{\mu\in\cP(\Rd)}\, \int_{\Rd} \frac{\Lg\psi}{\psi}\, \D{\mu}
\,=\, \sup_{\Rd} \frac{\Lg\psi}{\psi}\,<\,0\,.$$
Otherwise, we would have
$\Lg\psi+2\delta\psi\le 0$ for some $\delta>0$.
Applying It\^o's formula and the fact $\lim_{\abs{x}\to \infty} c(x)=0$ we obtain
\begin{equation*}
\psi(x)\,\ge\,
\Exp_x\Bigl[e^{\delta \uuptau_r}\psi(X_{\uuptau_r})\Bigr],
\quad \text{for large enough }\, r\,.
\end{equation*}
But the right hand side is unbounded, resulting in $\psi$ being unbounded.
This contradicts the fact $\psi\in\Cc^{2, +}_b(\Rd)$.
Thus in this case
\begin{equation*}
\lamstr(\Lg)\,<\, \inf_{\psi\in \Cc^{2, +}_b(\Rd)}\,
\sup_{\mu\in\cP(\Rd)}\, \int_{\Rd} \frac{\Lg\psi}{\psi}\, \D{\mu}\,.
\end{equation*}

On the other hand,
if $X$ is null-recurrent and $c\in\Cc_0(\Rd)$,
then $\lamstr(\Lg)$ cannot be nonzero if the principal
eigenfunction is bounded. 
For if $\Psi^*$ is bounded, then applying It\^o's formula it is easy to see that
\begin{equation*}
\Exp_x\biggl[\int_0^T
\bigl(c(X_t) \Psi^*(X_t)
-\lamstr(\Lg)\Psi^*(X_t)\bigr) \,\D{t} \biggr]\,=\,0\,.
\end{equation*}
Note that $\inf_{\Rd}\Psi^*>0$ by \cite[Lemma~2.1]{ari-anup},
since $\lamstr(\Lg)<0$.
Now divide both sides by $T$ and let
$T\to\infty$ to assert that $\lamstr(\Lg)=0$.
\end{remark}

\begin{remark}
\cref{T2.2} offers a variational formula for the principal eigenvalue
in the spirit of \cite{DoVa-76b}.
If we define
$\Lg_0 f(x)\df  a^{ij}(x)\,\partial_{ij} f(x)
+ b^{i}(x)\, \partial_{i} f(x)$, and the rate function
\begin{equation*}
I(\nu) \,\df\, - \inf_{f \in \Cc^{2, +}(\Rd)}\;
\int_\Rd\frac{\Lg_0 f}{f}\,d\nu\,,
\end{equation*}
then
\begin{equation*}
\lamstr(\Lg)\,=\,\sup_{\nu \in \cP(\Rd)}\left(\int_{\Rd} c(x)\,\nu(dx)- I(\nu)\right)\,.
\end{equation*}
\end{remark}

\begin{assumption}\label{A2.1}
The following hold.
\begin{itemize}
\item[(i)] There exists an inf-compact function $\ell\in\Cc(\Rd)$, and a positive
function $\Lyap\in\Sobl^{2, d}(\Rd)$, satisfying $\inf_{\Rd}\Lyap>0$, such that
\begin{equation*}
\sup_{u\in\Act}\,\cL_u \Lyap \,\le\,
\kappa_1 \Ind_{\cK}-\ell \Lyap \quad \text{in\ } \Rd\,,
\end{equation*}
for some constant $\kappa_1$ and a compact set $\cK$.
\item[(ii)]
The function $x\mapsto \beta\ell(x)-\max_{u\in\Act} c(x,u)$ is inf-compact
for some $\beta\in(0, 1)$.
\end{itemize}
\end{assumption}

As noted in \cite{ABS}, \cref{A2.1} does not hold for diffusions with bounded $a$,
and $b$. Therefore, to treat this case,
we consider an alternate set of conditions.

\begin{assumption}\label{A2.2}
The following hold.
\begin{itemize}
\item[(i)]
There exists a positive function $\Lyap\in\Sobl^{2,d}(\Rd)$,
satisfying $\inf_{\Rd} \Lyap > 0$, and a constant $\gamma>0$ such that
\begin{equation}\label{EA2.2A}
\sup_{u\in\Act}\,\cL_u \Lyap \,\le\, \kappa_1 \Ind_{\cK}-\gamma \Lyap
\quad \text{in\ } \Rd\,,
\end{equation}
for some constant $\kappa_1$ and a compact set $\cK$.
\item[(ii)]
$\norm{c^-}_\infty+\limsup_{\abs{x}\to\infty}\, \max_{u\in\Act}\,c(x,u)
<\gamma$.
\end{itemize}
\end{assumption}

The eigenvalue $\lamstr(\cG)$ in \cref{E-lamstr}
represents the optimal risk-sensitive ergodic cost
\cite{ari-anup,ABS,biswas-11a,biswas-11}.
In order to define this control problem, we need to introduce some additional notation.
For an admissible control $U$, the risk-sensitive criterion is defined as
\begin{equation*}
\sE(U) \,=\, \inf_{x\in\Rd}\; \limsup_{T\to\infty}\,\frac{1}{T}\,
\log\Exp^U_x\Bigl[\E^{\int_0^T c(X_s, U_s)\, \D{s}}\Bigr]\,.
\end{equation*}
The optimal value is defined as $\Lambda^*=\inf_{U\in\Uadm} \sE(U)$.

\begin{notation}
For a continuous function $g\,\colon\,\Rd\to(0,\infty)$
which is bounded below away from $0$,
we let $\order(g)$ denote the space of continuous functions
$f\colon\Rd\to\RR$ satisfying
$\sup_{x\in\Rd}\,\frac{\abs{f(x)}}{g(x)}<\infty$, and
by $\sorder(g)$ is subset consisting of those functions which satisfy
$\limsup_{R\to\infty}\,\sup_{x\in B_{R}^{c}}\,\frac{\abs{f(x)}}{g(x)}=0$.
\end{notation}

We borrow the following result from \cite{ABS}.

\begin{theorem}[{\cite[Theorem~4.1]{ABS}}]\label{T2.3}
Suppose that either \cref{A2.1}, or {\upshape\ref{A2.2}} holds.
Then $\Lambda^*=\lamstr(\cG)$, and for some
function $\Phi^*\in\Cc^{2,+}(\Rd)\cap\order(\Lyap^\beta)$,
for some $\beta\in (0, 1)$, we have
\begin{equation}\label{ET2.3A}
\cG\Phi^*(x)\,=\,\min_{u\in\Act}\,\bigl[\cL_u\Phi^*(x) + c(x,u)\Phi^*(x)\bigr]
\,=\,\lamstr(\cG)\Phi^*(x)\quad \text{in}\ \Rd\,.
\end{equation}
In addition, we have the following
\begin{itemize}
\item[\textup{(}i\textup{)}]
any measurable selector $v_*\colon\Rd\to\Act$
from the minimizer of \cref{ET2.3A} is an optimal Markov control
with respect to the risk-sensitive criterion;
\item[\textup{(}ii\textup{)}]
the function $\Phi^*$ has the stochastic representation
\begin{equation}\label{ET2.3B}
\Phi^*(x) \,=\, \Exp_x^{v_*}
\Bigl[\E^{\int_0^{\uuptau_r}(c(X_s,v_*(X_s))-\lamstr(\cG))\,\D{s}}
\Phi^*(X_{\uuptau_r})\Bigr]\quad \forall\,x\in B^c_r\,,
\end{equation}
for any $r>0$.
\item[\textup{(}iii\textup{)}]
$\Phi^*$ is the unique (up to a multiplicative constant) positive
solution of \cref{ET2.3A} in $\Cc^{2}(\Rd)$.
\end{itemize}
\end{theorem}

\begin{proof}
For the proof of this and related statements
 we refer to \cite[Theorems~4.1--4.3]{ABS}.
We provide a short proof of the fact that $\Phi^*\in\order(\Lyap^\beta)$
for the convenience of the reader.
We consider \cref{A2.2}.
Choose $r$ large enough so that for some suitable $\beta\in(0, 1)$
we have $\max_{u\in\Act} c(x,u)\le \beta \gamma$ for $x\in B_r^c$.
Without loss of generality we may assume $\cK\subset B_r$.
From the proof of \cite[Theorem~4.2]{ABS} it follows that for $x\in B^c_r$
we have
\begin{equation*}
\Phi^*(x) \,=\, \Exp_x^{v_*}
\Bigl[\E^{\int_0^{\uuptau_r}(c(X_s,v(X_s))-\lamstr(\cG))\,\D{s}}\,
\Phi^*(X_{\uuptau_r})\Bigr]\,,
\end{equation*}
which in turn, gives (since $\lamstr(\cG)=\Lambda^*\ge 0$)
\begin{equation*}
\Phi^*(x)\,\le\, \Exp_x^v\Bigl[\E^{\beta\gamma\uuptau_r} \Phi^*(X_{\uuptau_r})\Bigr]
\,\le\, \Exp_x^v\Bigl[\E^{\gamma\uuptau_r}
\bigl(\Phi^*(X_{\uuptau_r})\bigr)^{\nicefrac{1}{\beta}}\Bigr]^{\beta}
\,\le\,
\biggl[\max_{\partial B_r}\, \frac{\Phi^*}{\Lyap^\beta}\biggr]
\bigl(\Lyap(x)\bigr)^\beta\,,
\end{equation*}
where in the last inequality we use \cref{EA2.2A}.
The  proof under \cref{A2.1} is exactly analogous.
\end{proof}

We next state the Collatz--Wielandt formula for $\lamstr(\cG)$.

\begin{theorem}\label{T2.4}
Grant either \cref{A2.1}, or~{\upshape\ref{A2.2}}.
Then
\begin{align}
\lamstr(\cG) &\,=\, \sup_{\psi\in \Cc^{2,+}(\Rd)\cap\sorder(\Lyap)}\;
\inf_{\mu\in\cP(\Rd)}\; \int_{\Rd}\frac{\cG\psi}{\psi}\, \D{\mu}\label{ET2.4A}\\
&\,=\, \inf_{\psi\in \Cc^{2,+}(\Rd)}\;
\sup_{\mu\in\cP(\Rd)}\; \int_{\Rd}\frac{\cG\psi}{\psi}\, \D{\mu}\,.\label{ET2.4B}
\end{align}
\end{theorem}

\begin{remark}
The class of test functions $\psi$ in the representation \cref{ET2.4A} cannot,
in general, be enlarged to $\Cc^{2,+}(\Rd)$.
For a linear operator $\Lg$, it is known from
\cite[Theorem~1.4]{Berestycki-15} (see also \cref{T2.8} below)
that for any $\lambda\ge \lamstr(\Lg)$, there exists 
$\Psi\in\Cc^{2,+}(\Rd)$
satisfying 
\begin{equation*}
\Lg\Psi\,=\,\lambda\Psi\,.
\end{equation*}
Thus we obtain 
\begin{equation*}
\sup_{\psi\in\Cc^{2,+}(\Rd)}\; \inf_{\mu\in\cP(\Rd)}\;
\int_{\Rd}\frac{\Lg\psi}{\psi}\, \D{\mu}\,=\,\infty\,.
\end{equation*}
\end{remark}

The proof of \cref{T2.4} gives us the following maximum principle for the
semilinear operator $\cG$ in $\Rd$. This should be compared with
\cite[Theorem~1.6]{Berestycki-15}.

\begin{theorem}\label{T2.5}
Let either \cref{A2.1} or~{\upshape\ref{A2.2}} holds.
Let $\varphi\in\Cc^2(\Rd)\cap\sorder(\Lyap)$ satisfy
$\cG\varphi-\lamstr(\cG)\varphi\ge 0$ in $\Rd$,
and $\varphi(x_0)>0$ for some $x_0\in\Rd$.
Then  $\varphi=\kappa\Phi^*$ for some $\kappa>0$.
\end{theorem}

The following theorem could be seen as a refined maximum principle in $\Rd$.

\begin{theorem}\label{T2.6}
Let either \cref{A2.1} or~{\upshape\ref{A2.2}} holds.
Also suppose that $\lamstr(\cG)<0$.
Then for any $\varphi\in\Cc^2(\Rd)\cap\sorder(\Lyap)$
satisfying $\cG\varphi\,\ge\,0$ in $\Rd$
we have either $\varphi< 0$ or $\varphi=0$
in $\Rd$.
\end{theorem}

The next result provides another characterization of $\lamstr(\cG)$.

\begin{theorem}\label{T2.7}
Under either \cref{A2.1} or~{\upshape\ref{A2.2}}, we have
\begin{align*}
\lamstr(\cG) &\,=\,\lambda''(\cG)
\,=\,\inf\,\Bigl\{\lambda\in\RR\,\colon \exists \psi\in\Cc^2(\Rd),\;
\inf_{\Rd}\psi>0,\; \\
&\mspace{250mu} \text{satisfying}\
\cG\psi-\lambda\psi\le 0 \text{\ a.e.\ in\ } \Rd\Bigr\}\,.
\end{align*}
\end{theorem}

We next, prove the existence of infinitely many generalized eigenvalues for the
semilinear operator $\cG$.
For the linear operator $\Lg$, it has been recently
shown in \cite[Theorem~1.4]{Berestycki-15} that for any $\lambda\geq \lamstr(\Lg)$
there exists a positive $\Psi\in\Sobl^{2, d}(\Rd)$ satisfying
$\Lg \Psi=\lambda\Psi$.
Our next result is in the same spirit but for the semilinear operator $\cG$.

\begin{theorem}\label{T2.8}
For any $\lambda\ge\lamstr(\cG)$ there exists a positive
$\Phi_\lambda\in\Cc^2(\Rd)$ 
satisfying
\begin{equation*}
\cG\Phi_\lambda\,=\,\lambda\Phi_\lambda\quad \text{in}\ \Rd.
\end{equation*}
\end{theorem}

It is straightforward to show that we have an analogous version of all the
preceding results for $\cH$ (see \eqref{E-Lh}).
Let $\lamstr(\cH)$ be the corresponding principal eigenvalue
defined as in \cref{E-lamstr}.
It is clear from the definition that $\lamstr(\cG)\le \lamstr(\cH)$.
We present the following result.
A similar result is known for Dirichlet problems in bounded
domains \cite[Theorem~1.8]{Quaas-08a}.

\begin{theorem}\label{T2.19}
Let either \cref{A2.1} or~{\upshape\ref{A2.2}} hold, and
suppose $\lamstr(\cG)< \lamstr(\cH)$.
Then for any $\lambda\in(\lamstr(\cG), \lamstr(\cH))$,
there exists no non-trivial solution
of $\cG\varphi-\lambda\varphi=0$ for some $\varphi\in\Cc^2(\Rd)\cap\sorder(\Lyap)$.
\end{theorem}

\section{Proofs}\label{S3}

We begin with the proof of \cref{T2.1}.
Let us point out that we use several results from \cite{Quaas-08a}
which deals with operators that are convex in
$(\grad^2\psi, \grad\psi, \psi)$.
Since $\cG$ is concave in $(\grad^2\psi, \grad\psi, \psi)$,
 we can apply the results of \cite{Quaas-08a} with suitable modification.

\begin{proof}[Proof of \cref{T2.1}]
Throughout the proof, we let $\lambda_D\equiv\lambda_D(\cG)$.
We claim that for any $\psi\in\Cc^{2,+}_0(D)$ we have
\begin{equation}\label{ET2.1D}
\inf_{D}\;\frac{\cG\psi}{\psi}\,\le\, \lambda_D\,.
\end{equation}
Arguing by contradiction, suppose that for some $\lambda>\lambda_D$
it holds that
\begin{equation*}
\inf_{D}\;\frac{\cG\psi}{\psi}\,=\,\lambda\,.
\end{equation*}
Thus we have 
\begin{equation*}
\cG \psi \,\ge\, \lambda \psi \quad \text{in}\ D\,,
\quad \psi>0\quad \text{in}\ D\,, \quad \psi=0 \quad \text{on}\; \partial D\,.
\end{equation*}
Then $\psi=t \Psi$ for some $t>0$ by \cite[Theorem~1.2]{Quaas-08a},
where $\Psi$ denotes the principal eigenvector in \cref{E-Psi}.
This implies that $\cG\Psi\ge \lambda \Psi$, thus
leading to a contradiction, and establishing the claim in \cref{ET2.1D}.

By \cref{ET2.1D}, we obtain
\begin{equation*}
\sup_{\psi\in\Cc^{2,+}(D)}\; \inf_{\mu\in\cP(D)}\;
\int_{D}\frac{\cG\psi}{\psi}\, \D{\mu}
\,=\, \sup_{\psi\in\Cc^{2,+}(D)}\; \inf_{D}\; \frac{\cG\psi}{\psi}
\,\le\, \lambda_D\,.
\end{equation*}
On the other hand, choosing $\Psi$ as a test function,
we have from \cref{E-Psi} that
\begin{equation*}
\sup_{\psi\in\Cc^{2,+}(D)}\; \inf_{\mu\in\cP(D)}\;
\int_{D}\frac{\cG\psi}{\psi}\,\D{\mu}
\,\ge\, \inf_{\mu\in\cP(D)}\; \int_{D}\frac{\cG\Psi}{\Psi}\, \D{\mu}
\,=\, \lambda_D\,.
\end{equation*}
This proves \cref{ET2.1A}.

Now we consider \cref{ET2.1B}.
We clearly have
\begin{equation*}
\inf_{\psi\in\Cc^{2,+}(D)}\; \sup_{\mu\in\cP(D)}\;
\int_{D}\frac{\cG\psi}{\psi}\,\D{\mu}
\,\le\, \sup_{\mu\in\cP(D)}\; \int_{D}\frac{\cG\Psi}{\Psi}\, \D{\mu}
\,=\, \lambda_D\,.
\end{equation*}
To get the opposite inequality, we apply the characterization in \cref{E-lamD}.
Note that it follows from \cref{E-lamD} that for any $\psi\in\Cc^{2,+}(D)$
we have
$\sup_D \frac{\cG\psi}{\psi}\ge \lambda_D$,
and hence,
\begin{equation*}
\inf_{\psi\in\Cc^{2,+}(D)}\; \sup_{\mu\in\cP(D)}\;
\int_{D}\frac{\cG\psi}{\psi}\, \D{\mu}
\,=\, \inf_{\psi\in\Cc^{2,+}(D)}\; \sup_{D}\;
\frac{\cG\psi}{\psi}\,\ge\, \lambda_D\,.
\end{equation*}
This establishes \cref{ET2.1B},
and completes the proof.
\end{proof}

\begin{remark}
\Cref{ET2.1A,ET2.1B} hold for a more general class of operators.
More precisely, if $\cG(\grad^2\psi, \grad\psi, \psi, x)$ is a general nonlinear 
elliptic operator that is concave in first three arguments  and satisfies
the assumptions $(H_0)$--$(H_2)$ in \cite{Quaas-08a}, we still have a
Collatz--Wielandt formula for the
eigenvalue $\lambda_D$. The proof follows from the arguments in
the proof of \cref{T2.1}.
In particular, if we consider the operator
\begin{equation*}
\cG(\grad^2\psi, \grad\psi, \psi, x) \,=\, a^{ij}(x)\,\partial_{ij} \psi(x)
+ b^{i}(x)\, \partial_{i} \psi(x) + c(x)\psi\,,
\end{equation*}
where $b$, and $c$ are bounded Borel measurable functions, and $a$ is continuous
and satisfies (A3), then
we have the Collatz--Wielandt representation in \cref{ET2.1A}-\cref{ET2.1B}
for its principal eigenvalue.
\end{remark}

\begin{proof}[Proof of \cref{T2.18}]
Using the eigenvalue equation for $\cH$, analogous to \eqref{E-Psi},
it is easily seen that
\begin{equation*}
\sup_{\mu\in\cP(D)}\; \inf_{\psi\in\Cc^{2,+}(D)}\;
\int_D \frac{\cH\psi}{\psi}\, \D{\mu}\,\le\, \lambda_D(\cH)\,.
\end{equation*}
To show the reverse inequality we consider a smooth domain $D_n\Subset D$.
Define  $\Tilde{\cH}$ as
\begin{equation*}
\Tilde\cH\psi \,=\, \cH\psi + \langle \grad\psi, a\grad\psi\rangle\,.
\end{equation*}
Note that $\psi\in \Cc^{2}(D)\mapsto \Tilde\cH\psi$ is convex.
 Therefore,
\begin{equation}\label{ET2.1E}
\begin{aligned}
\sup_{\mu\in\cP(D)}\;
\inf_{\psi\in\Cc^{2,+}(D)}\int_D \frac{\cH\psi}{\psi}\, \D{\mu} 
&\,\ge\, \sup_{\mu\in\cP(\Bar{D}_n)}\;
\inf_{\psi\in\Cc^{2,+}(D)}\int_D \frac{\cH\psi}{\psi}\, \D{\mu}  \\
&\,=\, \sup_{\mu\in\cP(\Bar{D}_n)}\;
\inf_{\psi\in\Cc^{2}(D)}\int_D \Tilde{\cH}\psi\, \D{\mu} \\
&\,=\,  \inf_{\psi\in\Cc^{2}(D)}\;
\sup_{\mu\in\cP(\Bar{D}_n)}\int_D \Tilde{\cH}\psi\, \D{\mu} \\
&\,=\,  \inf_{\psi\in\Cc^{2,+}(D)}\;
\sup_{\mu\in\cP(\Bar{D}_n)}\int_D \frac{\cH\psi}{\psi}\, \D{\mu}  \\
&\,=\,  \inf_{\psi\in\Cc^{2,+}(D)}\; \max_{\Bar{D}_n}\,\frac{\cH\psi}{\psi}\,,
\end{aligned}
\end{equation}
where in the third line we used Sion's minimax theorem \cite{Sion-58}.
In view of \cite[Theorem~1.1]{Quaas-08a} we have 
\begin{equation*}
\max_{\Bar{D}_n}\,\frac{\cH\psi}{\psi} \,\ge\,
\lambda_{D_n}(\cH)\quad \text{for all}\; \psi\in \Cc^{2,+}(D)\,,
\end{equation*}
and therefore, combining with \eqref{ET2.1E} we obtain
\begin{equation*}
\sup_{\mu\in\cP(D)}\; \inf_{\psi\in\Cc^{2,+}(D)} \,\ge\, \lambda_{D_n}(\cH)\,.
\end{equation*}
Now let $n\to\infty$, so that $D_n\uparrow D$ and use the fact
$\lambda_{D_n}(\cH)\to \lambda_D(\cH)$ \cite[Proposition~4.10]{Quaas-08a}
to deduce that
\begin{equation*}
\sup_{\mu\in\cP(D)}\;
\inf_{\psi\in\Cc^{2,+}(D)}\int_D \frac{\cH\psi}{\psi}\, \D{\mu}
\,\ge\, \lambda_{D}(\cH)\,,
\end{equation*}
thus establishing \eqref{ET2.18A}.
\end{proof}

\begin{proof}[Proof of \cref{T2.2}]
We  distinguish two cases.
First, suppose that $\lamstr(\Lg) >0$.
In this case we claim that the principal eigenfunction $\Phi^*$ is bounded.
For this, we argue as follows.
The principal eigenfunction is obtained as a limit of Dirichlet principal
eigenfunction over balls. Again, if $(\lambda_n, \Psi_n)$ is 
the principal eigenpair in $B_n(0)$ then it follows from \cite[Lemma 2.2]{ABS} that
\begin{equation*}
\Psi_n(x)\,=\, \Exp_x\Bigl[\E^{\int_0^{\uuptau_r}
[c(X_s)-\lambda_n]\, \D{s}}\,\Psi_n(X_{\uuptau_r})
\,\Ind_{\{\uuptau_r<\uptau_n\}}\Bigr], \quad x\in B_n\setminus B^c_r(0)\,.
\end{equation*}
for any $n>r>0$.
Since $\lambda_n\to \lamstr(\Lg)$, and $c$ vanishes at infinity,
it follows that for some large enough $r$ and $n_0\in\NN$, we have
$\sup_{x\in B_r^c}\, c(x) -\lambda_n<0$ for all $n\ge n_0$.
This shows that
$\sup_{B_n}\Psi_n=\sup_{B_r} \Psi_n$ for all $n\ge n_0$.
Thus, the limit of $\Psi_n$ is also bounded, proving the claim.

Therefore, we have
\begin{equation*}
\lamstr(\Lg) \,\le\, \ulamstr \,\df\,\sup_{\psi\in\Cc_b^{2,+}(\Rd)}\;
\inf_{\mu\in\cP(\Rd)}\;\int_{\Rd}\frac{\Lg\psi}{\psi}\, \D{\mu}\,.
\end{equation*}
Suppose $\psi\in\Cc_b^{2,+}(\Rd)$ satisfies
$\Lg\psi-(\lamstr(\Lg)+\epsilon)\psi\ge 0$
for some $\epsilon>0$.
Let $r>0$ be such that $\sup_{x\in B_r^c}\,c(x)<\lamstr(\Lg)+\varepsilon$.
It is fairly straightforward to show that $\psi$ satisfies
\begin{equation*}
\psi(x) \,\le \,
\Exp_x\Bigl[\E^{\int_0^{\uuptau_r}(c(X_s)-\lamstr(\Lg)-\epsilon)\,\D{s}}
\psi(X_{\uuptau_r})\Bigr]\quad \forall\,x\in B^c_r\,.
\end{equation*}
Indeed, since
\begin{equation*}
\Exp_x\Bigl[\E^{\int_0^{\uptau_R}(c(X_s)-\lamstr(\Lg)-\epsilon)\,\D{s}}
\psi^*(X_{\uptau_R})\,\Ind_{\{\uptau_R<\uptau_r\}}\Bigr]
\,\le\,
\norm{\psi}_\infty\, \Prob_x\bigl(\uptau_R<\uptau_r\bigr)
\quad \forall\,x\in B^c_r\,.
\end{equation*}
it follows that this quantity tends to $0$ as $R\to\infty$.

Thus  $C\Phi^*\ge\psi$ for some positive constant $C$,
and this is clearly impossible by the strong maximum principle.
It follows that $\lamstr(\Lg) = \ulamstr$.

On the other hand, as shown in \eqref{E-general}, we have
\begin{equation}\label{PT2.2B}
\lamstr(\Lg) \,=\, \olamstr \,\df\,\inf_{\psi\in\Cc_b^{2,+}(\Rd)}\;
\sup_{\mu\in\cP(\Rd)}\; \int_{\Rd}\frac{\Lg\psi}{\psi}\, \D{\mu}\,,
\end{equation}
and the proof of \cref{T2.18} gives us
\begin{equation}\label{PT2.2C}
\lamstr(\Lg) \,=\, \sup_{\mu\in\cP(\Rd)}\;\inf_{\psi\in\Cc_b^{2,+}(\Rd)}\;
 \int_{\Rd}\frac{\Lg\psi}{\psi}\, \D{\mu}\,.
\end{equation}

Next, suppose that $\lamstr(\Lg)=0$. Note that by definition we have
\begin{equation*}
\sup_{\psi\in\Cc_b^{2, +}(\Rd)}\, \inf_{\Rd}\,\frac{\Lg\psi}{\psi}\,\le\,
 \lambda'(\Lg)\,.
\end{equation*}
 Again by \cite[Theorem~1.9\,(iii)]{Berestycki-15}
we have $\lamstr(\Lg)=\lambda'(\Lg)=0$. This gives us
\begin{equation*}
\sup_{\psi\in\Cc_b^{2,+}(\Rd)}\,\inf_{\mu\in\cP(\Rd)}\,
\int_{\Rd}\frac{\Lg\psi}{\psi} \D{\mu}\,\le\, 0\,.
\end{equation*}
Since $\lambda'(\Lg)=0$, then by definition,
for any $\varepsilon>0$ there exists $\psi\in \Cc^{2, +}_b(\Rd)$ satisfying
$\Lg \psi \geq -\varepsilon\psi$,
which implies that
\begin{equation*}
\frac{\Lg\psi}{\psi}\,\ge\, -\varepsilon\,.
\end{equation*}
Thus 
\begin{equation*}
\inf_{\mu\in\cP(\Rd)}\,\int_{\Rd}\,\frac{\Lg\psi}{\psi} \D{\mu}
\,\ge\, -\varepsilon\,,
\end{equation*}
and since $\varepsilon>0$ is arbitrary, we obtain
\begin{equation*}
\sup_{\psi\in\Cc_b^{2, +}(\Rd)}\,\inf_{\mu\in\cP(\Rd)}\,
\int_{\Rd}\frac{\Lg\psi}{\psi} \D{\mu}\,=\, 0\,.
\end{equation*}

Lastly, consider the case $\lamstr(\Lg)<0$.
Then it is easy to show that $\Phi^*$ is bounded away from $0$.
Hence $\lamstr(\Lg)=\lambda''(\Lg)$.
By (A1)--(A3) and \cite[Theorem~1.7\,(iii)]{Berestycki-15},
we have $\lambda''(\Lg)\ge\lambda'(\Lg)$.
Since $\lamstr(\Lg)=\lambda''(\Lg)$, we have
$\lamstr(\Lg)\ge\lambda'(\Lg)$, which implies that
$\lamstr(\Lg)\ge\ulamstr$.
On the other hand, we have
$\ulamstr\ge\lamstr(\Lg)$ \cite[Theorem~1.7\,(ii)]{Berestycki-15}.
Thus $\ulamstr=\lamstr(\Lg)$.

We leave it to the reader to verify that \cref{PT2.2B,PT2.2C}
hold if we replace $\Cc_b^{2,+}(\Rd)$ with $\Cc^{2,+}(\Rd)$,
and that this is also true in the case $\lamstr(\Lg)\ge0$.

It remains to consider $\cG$.
Suppose $\lamstr(\cG)>0$.
As before, the corresponding principal eigenfunction is bounded.
Therefore, the 
second equality follows from the proof of \cref{E-general}. Moreover, 
$$\lamstr(\cG) \,\le \,\sup_{\psi\in\Cc_b^{2,+}(\Rd)}\,
\inf_{\mu\in\cP(\Rd)}\, \int_{\Rd}\frac{\cG\psi}{\psi}\, \D{\mu}\,.$$
Let $v_*$ be a minimizing selector of 
$$\cG\Phi^* \,=\, \lamstr(\cG)\Phi^*\,,$$
and recall that the associated process is recurrent.
Then, denoting the corresponding generator by $\Lg_{v_*}$ and
applying the previous result, we note that
\begin{align*}
\lamstr(\cG)\,=\,\lamstr(\Lg_{v_*})
&\,\ge\,
\sup_{\psi\in\Cc_b^{2,+}(\Rd)}\,
\inf_{\mu\in\cP(\Rd)}\, \int_{\Rd}\frac{\Lg_{v_*}\psi}{\psi}\, \D{\mu} \\
&\,\ge\, \sup_{\psi\in\Cc_b^{2,+}(\Rd)}\;
\inf_{\mu\in\cP(\Rd)}\; \int_{\Rd}\frac{\cG\psi}{\psi}\, \D{\mu}\,.
\end{align*}
Combining we have
$$\lamstr(\cG) \,= \,\sup_{\psi\in\Cc_b^{2,+}(\Rd)}\,
\inf_{\mu\in\cP(\Rd)}\, \int_{\Rd}\frac{\cG\psi}{\psi}\, \D{\mu}\,.$$
This completes the proof.
\end{proof}

We need the following lemma for the proofs of \cref{T2.4,T2.5}.

\begin{lemma}\label{L3.1}
Grant \cref{A2.1} or \cref{A2.2}.
Suppose that the function
$\psi\in\Sobl^{2,d,+}(\Rd)\cap\sorder(\Lyap)$ satisfies
\begin{equation}\label{EL3.1A}
\Lg_v\psi + c(x, v(x))\psi\,\ge\, \lambda\psi\quad \text{a.e. in}\ \Rd\,,
\end{equation}
for some Markov control $v$, and $\lambda\in\RR$.
Then there exists $r_\circ>0$ not depending on $\psi$ such that
\begin{equation}\label{EL3.1B}
\psi(x)\,\le\,\Exp_x^v
\Bigl[\E^{\int_0^{\uuptau_r}(c(X_s,v(X_s))-\lambda)\,\D{s}}
\,\psi(X_{\uuptau_r})\Bigr]\quad \text{for}\ x\in B^c_r\,,
\end{equation}
for all $r\ge r_\circ$.
In addition, if \cref{EL3.1A} holds with equality, then we have
equality in \cref{EL3.1B}.
\end{lemma}

\begin{proof}
We only consider the case of \cref{A2.1}.
The proof under \cref{A2.2} is completely analogous.
Choose $r$ large enough so that $\max_{u\in\Act}(c(x,u)-\lambda)< \ell(x)$
in $B^c_r$ and $\cK\subset B_r$.
For any $R>r>0$, we have
\begin{equation}\label{PL3.1D}
\begin{aligned}
\psi(x) &\,\le\,
\Exp_x^v\Bigl[\E^{\int_0^{\uuptau_r}(c(X_s,v(X_s))-\lambda)\,\D{s}}\,
\psi(X_{\uuptau_r})\Ind_{\{\uuptau_r<\uptau_R\}}\Bigr]\\
&\mspace{180mu}
+ \Exp_x^v\Bigl[\E^{\int_0^{\uptau_R}(c(X_s,v(X_s))-\lambda)\,\D{s}}\,
\psi(X_{\uptau_R})\Ind_{\{\uuptau_r>\uptau_R\}}\Bigr]\,.
\end{aligned}
\end{equation}
We first estimate the limit of the second term of \cref{PL3.1D} as $R\to\infty$. 
\begin{equation}\label{PL3.1E}
\begin{aligned}
\Exp_x^v\Bigl[& \E^{\int_0^{\uptau_R} (c(X_s,v(X_s))-\lambda)\,\D{s}}\,
\psi(X_{\uptau_R})\Ind_{\{\uuptau_r>\uptau_R\}}\Bigr]\\
&\mspace{100mu}\,\le\, \biggl(\max_{\partial B_R}\,\frac{\abs{\psi}}{\Lyap}\biggr)
\Exp_x^v\Bigl[\E^{\int_0^{\uptau_R}(c(X_s,v(X_s))-\lambda)\,\D{s}}\,
\Lyap(X_{\uptau_R})\Ind_{\{\uuptau_r>\uptau_R\}}\Bigr]\\
&\mspace{100mu}\,\le\, \biggl(\max_{\partial B_R}\,\frac{\abs{\psi}}{\Lyap}\biggr)
\Exp_x^v\Bigl[\E^{\int_0^{\uptau_R}\ell(X_s)\,\D{s}}\,
\Lyap(X_{\uptau_R})\Ind_{\{\uuptau_r>\uptau_R\}}\Bigr]\\
&\mspace{100mu} \,\le\,
\biggl(\max_{\partial B_R}\,\frac{\abs{\psi}}{\Lyap}\biggr) \Lyap(x)
\,\xrightarrow[R\to\infty]{}\, 0\,,
\end{aligned}
\end{equation}
where in the last line we use that fact that $\psi\in\sorder(\Lyap)$.
Thus letting $R\to\infty$ in \cref{PL3.1D}, and using the
monotone convergence theorem, we obtain \cref{EL3.1B}.
The last sentence is evident from \cref{PL3.1D,PL3.1E}.
This completes the proof.
\end{proof}

We continue with the proof of \cref{T2.4}.

\begin{proof}[Proof of \cref{T2.4}]
Throughout this proof $\lamstr\equiv\lamstr(\cG)$.
Since $\Phi^*\in\order(\Lyap^\beta)$ by \cref{T2.3}, it follows that
\begin{equation}\label{ET2.4C}
\lamstr \,\le\, \sup_{\psi\in \Cc^{2,+}(\Rd)\cap\sorder(\Lyap)}\;
\inf_{\mu\in\cP(\Rd)}\; \int_{\Rd}\frac{\cG\psi}{\psi}\, \D{\mu}\,.
\end{equation}
We claim that for any $\psi\in\Cc^{2,+}(\Rd)\cap\sorder(\Lyap)$ we have
\begin{equation}\label{ET2.4D}
\inf_{\Rd}\;\frac{\cG\psi}{\psi} \,\le\, \lamstr\,.
\end{equation}
Indeed, suppose to the contrary that for some $\lambda>\lamstr$ it holds that
\begin{equation*}
\inf_{\Rd}\;\tfrac{\cG\psi}{\psi}\,=\,\lambda\,.
\end{equation*}
This implies that
\begin{equation*}
\cG\psi -\lambda\psi\,\ge\, 0\quad \text{in}\ \Rd\,.
\end{equation*}
Let $v_*$ be a measurable selector of the HJB in \cref{T2.3}.
To simplify the notation we let $c_{v_*}(x)\df c\bigl(x,v_*(x)\bigr)$.
Then we have
\begin{equation}\label{ET2.4E}
\cL_{v_*}\psi+ \bigl(c_{v_*}-\lambda\bigr)\psi\,\ge\, 0\,.
\end{equation}
By \cref{L3.1,T2.3} we have
\begin{equation}\label{ET2.4F}
\psi(x) \,\le\, \Exp^{v_*}_{x}
\Bigl[\E^{\int_0^{\breve\uptau_r} (c_{v_*}(X_s)-\lambda)\,\D{s}} \,
\psi(X_{\breve\uptau_r})\Bigr], \quad x\in B_r^c\,,
\end{equation}
and
\begin{equation}\label{ET2.4H}
\Phi^*(x) \,=\, \Exp^{v_*}_{x}\Bigl[\E^{\int_0^{\breve\uptau_r} (c_{v_*}(X_s)
-\lamstr)\D{s}} \,\Phi^*(X_{\breve\uptau_r})\Bigr], \quad x\in B_r^c\,,
\end{equation}
respectively.
Let $\kappa=\max_{\bar{B_r}} \frac{\psi}{\Phi^*}$.
Then from \cref{ET2.4F,ET2.4H} we see that
$\psi\le\kappa \Phi^*$ in $\Rd$, and for some $\abs{x_0}\le r$ we have 
$\varphi(x_0)-\kappa\Phi^*(x_0)=0$.
Since 
\begin{equation*}
\cL_{v_*}\Phi^*+ (c_{v_*}-\lamstr)\Phi^* \,=\, 0\,,
\end{equation*}
using \cref{ET2.4E} we obtain
\begin{equation*}
\cL_{v_*}(\kappa\Phi^*-\psi) - (c_{v_*}-\lamstr)^-(\kappa\Phi^*-\psi)
\,\le\, 0\quad \text{in}\ \Rd\,.
\end{equation*}
It follows by the strong maximum principle that $\kappa\Phi^*=\psi$,
and this contradicts \cref{ET2.4E} since $\lambda>\lamstr$.
This proves \cref{ET2.4D}.

Now using \cref{ET2.4D} we obtain
\begin{equation*}
\sup_{\psi\in\Cc^{2,+}(\Rd)\cap\sorder(\Lyap)}\; \inf_{\mu\in\cP(\Rd)}\;
\int_{\Rd}\frac{\cG\psi}{\psi}\, \D{\mu} \,\le\,
\sup_{\psi\in\Cc^{2,+}(\Rd)\cap\sorder(\Lyap)}\; \inf_{\Rd}\;
\frac{\cG\psi}{\psi} \,\le\, \lamstr\,.
\end{equation*}
Hence, using \cref{ET2.4C}, we obtain \cref{ET2.4A} .

From \cref{E-lamstr} it is easily seen that
\begin{equation*}
\sup_{\Rd}\, \frac{\cG\psi}{\psi}
\,\ge\, \lamstr\quad \text{for any}\ \psi\in\Cc^{2,+}(\Rd)\,,
\end{equation*}
and therefore, 
\begin{equation*}
\sup_{\mu\in\cP(\Rd)}\; \int_{\Rd}\frac{\cG\psi}{\psi}\; \D{\mu}
\,\ge\, \lamstr \quad \text{for any}\ \psi\in\Cc^{2,+}(\Rd)\,.
\end{equation*}
This gives us
\begin{equation*}
\inf_{\psi\in\Cc^{2,+}(\Rd)}\; \sup_{\mu\in\cP(\Rd)}\;
 \int_{\Rd}\frac{\cG\psi}{\psi}\, \D{\mu} \,\ge\, \lamstr\,.
\end{equation*}
Now choosing $\psi=\Phi^*$ in the above display,
we get equality which proves \cref{ET2.4B}.
\end{proof}

The function space used in the representation \eqref{ET2.4A} can be extended to
$\cA_\Lyap\df\Cc^{2, +}(\Rd)\cap\order(\Lyap)$, provided we impose certain
assumptions on the Lyapunov function $\Lyap$.
This is the subject of the following theorem.

\begin{theorem}
Suppose that any one of the following is true.
\begin{itemize}
\item[\textup{(}a\textup{)}]
Assumption~\ref{A2.1}\,(i) holds with an inf-compact function $\Lyap$
and the function 
\begin{equation*}
x\mapsto \ell(x)-\frac{\langle \grad\Lyap(x), a(x)\grad\Lyap\rangle}{\Lyap^2(x)
\log\Lyap(x)}-\max_{u\in\Act}\, c(x,u)\quad \text{is inf-compact}.
\end{equation*}
\item[\textup{(}b\textup{)}]
Assumption~\ref{A2.2} holds with an inf-compact function $\Lyap$ and 
\begin{equation*}
\lim_{\abs{x}\to\infty}\frac{\langle \grad\Lyap(x), a(x)\grad\Lyap\rangle}
{\Lyap^2(x)\log\Lyap(x)} \,=\, 0\,.
\end{equation*}
\end{itemize}
Then we have
\begin{equation*}
\lamstr(\cG) \,=\, \sup_{\psi\in \cA_\Lyap}\;
\inf_{\mu\in\cP(\Rd)}\; \int_{D}\frac{\cG\psi}{\psi}\, \D{\mu}\,.
\end{equation*}
\end{theorem}

\begin{proof}
From \cite[Theorems~4.1 and~4.2]{ABS} we note that parts (i)--(ii) of
\cref{T2.3} hold under the above assumptions. 
Using \eqref{ET2.3B} it is
easily seen that $\Phi^*\in\order(\Lyap)$.
Now define 
\begin{equation*}
\Tilde\Lyap \,\df\, \Lyap \log\Lyap\,, \quad \Tilde\ell \,\df\,
\ell-\frac{\langle \grad\Lyap(x), a(x)\grad\Lyap\rangle}{\Lyap^2(x)\log\Lyap(x)}\,.
\end{equation*}
 Then an easy calculation gives
\begin{align*}
\max_{u\in\Act}\, \cL_u \Tilde\Lyap 
&\,\le\, \kappa_1 (\log\Lyap)  \Ind_\cK-\ell\Tilde\Lyap
+ \kappa_1 \Ind_\cK-\ell \Lyap
+ \frac{1}{\Lyap}\langle \grad \Lyap, a\grad \Lyap\rangle\\
&\,\le\, \kappa\Bigl(\max_{\cK}\, \Lyap + 1\Bigr) \Ind_{\cK}
-\Tilde\ell\Tilde\Lyap.
\end{align*}
Therefore $\Tilde\Lyap$ can be used as a new Lyapunov function pay-off function
$\Tilde\ell$.
Again, $\Lyap$ being inf-compact we have
$\order(\Lyap)\subset\sorder(\Tilde\Lyap)$.
Hence for any function $\psi$ satisfying \eqref{ET2.4E} the
estimate in \eqref{ET2.4F} holds.
Then rest of the proof follows from Theorem~\ref{T2.4}.
\end{proof}

The proof of \cref{T2.5} which follows, uses an argument similar 
to the one used in \cref{T2.4}.

\begin{proof}[Proof of \cref{T2.5}]
It is given that $\varphi(x_0)>0$.
Without loss of generality we may assume that $x_0=0$
and $\varphi>0$ in $B_\delta(0)$ for some $\delta>0$.
Choose a stable optimal Markov policy $v_*$ from \eqref{ET2.3A} as in the
proof of \cref{T2.4}.
By \cref{L3.1} we have the stochastic representation in \cref{ET2.4F}
for all large enough $r>0$.
Let $\kappa=\max_{\bar B_r}\frac{\varphi^+}{\Phi^*}$.
Note that $\kappa>0$, since $\varphi>0$ in $B_\delta(0)$.
It now follows from \cref{ET2.3B,ET2.4F}
that $\varphi\le \kappa\Phi^*$ in $\Rd$, and for some $y_0\in \Bar{B}_r$
we have $\varphi(y_0)=\kappa\Phi^*(y_0)$.
Combining the inequalities
\begin{equation*}
\cL_{v_*}\Phi^*+ \bigl(c_{v_*}-\lamstr(\cG)\bigr)\Phi^*
\,=\, 0\,,\quad \text{and}
\quad \cL_{v_*}\varphi+ \bigl(c_{v_*}-\lamstr(\cG)\bigr)\varphi \,\ge\, 0\,,
\end{equation*}
we obtain
\begin{equation*}
\cL_{v_*}(\kappa\Phi^*-\varphi)
- \bigl(c_{v_*}-\lamstr(\cG)\bigr)^-(\kappa\Phi^*-\varphi)
\,\le\, 0\quad \text{in}\ \Rd\,.
\end{equation*}
Therefore, $\kappa\Phi^*=\varphi$ in $\Rd$
by the strong maximum principle. This completes the proof.
\end{proof}

We continue with the proof of \cref{T2.6}.

\begin{proof}[Proof of \cref{T2.6}]
To the contrary, suppose that $\varphi(x_0)>0$. Without loss of generality we may
assume that $x_0=0$
and $\varphi>0$ in $B_\delta(0)$ for some $\delta>0$.
Choosing an optimal stable control $v_*$  we deduce, as in the proof
of \cref{T2.5}, that for
some positive $\kappa$ we have
$\kappa\Phi^*-\varphi\ge 0$ and the minimum value $0$ is attained at some
some point $y_0$.
Denote by $\xi=\frac{\varphi}{\kappa\Phi^*}$.
An easy calculation gives
\begin{equation*}
\cL_{v_*} \xi + \langle b + 2 a\grad(\log \Phi^*), \grad \xi\rangle
+ \lamstr(\cG) \xi\,\ge\,
 0\,, \quad \text{in}\; \Rd\,.
\end{equation*}
Note that $\xi\le 1$ and $\xi(y_0)=1$.
Thus by the strong maximum principle we have $\xi=1$, implying
that $\varphi=\kappa\Phi^*$.
But this is not possible as $\lamstr(\cG)<0$.
Hence we must have $\varphi\le 0$.
The result follows
by another application of the strong maximum principle.
\end{proof}

To prove \cref{T2.7} we first consider an eigenvalue problem for a perturbed $c$.
For Assumption~(a) in \cref{T2.7} we define
\begin{equation*}
c_m(x,u) \,=\, c(x,u)+\frac{1}{m}\ell(x) \quad \text{for}\ x\in\Rd\,, \ m\ge 1\,.
\end{equation*}
For Assumption~(b) in \cref{T2.7} and $m\ge 1$ we consider a smooth function
$\zeta_m\colon\Rd\to[0, 1]$, satisfying $\zeta_m(x)=1$ in $B_m$
and $\zeta_m(x)=0$ in $B^c_{m+1}$, and define
\begin{equation}\label{E-cm}
c_m(x,u) \,\df\, \zeta_m(x) c(x,u)
+ (1-\zeta_m(x))\Bigl(\delta+\limsup_{\abs{z}\to\infty}\,
\max_{u\in\Act}\,c(z,u)\Bigr)\,,
\end{equation}
where $\delta$ is small enough to satisfy 
\begin{equation*}
\delta\,<\,\gamma - \norm{c^-}_\infty
- \limsup_{\abs{z}\to\infty}\, \max_{u\in\Act}\,c(z,u)\,.
\end{equation*}
Then following an argument similar to \cite[Lemmas~3.4 and~3.5]{BS18}
we can establish the following.

\begin{lemma}\label{L3.2}
Grant the assumptions of \cref{T2.7}.
Then there exists a unique $\Psi^*_m\in\Cc^{2,+}(\Rd)$ with $\inf_{\Rd}\Psi^*_m>0$,
satisfying
\begin{equation*}
\begin{aligned}
\cG_m \Psi^*_m(x)
&\,\df\, a^{ij}(x)\,\partial_{ij} \Psi^*_m(x)
+ \min_{u\in\Act}\,\Bigl[b^{i}(x,u)\, \partial_{i} \Psi^*_m(x)
+ c_m(x,u) \Psi^*_m(x)\Bigr]\\
&\,=\,\lamstr(\cG_m)\Psi^*_m(x)\quad\forall\,x\in\Rd\,.
\end{aligned}
\end{equation*}
In addition, we have
\begin{equation*}
\lamstr(\cG_m) \,=\, \inf_{U\in\Uadm}\, \limsup_{T\to\infty}\, \frac{1}{T}\,
\log\Exp^U_x\Bigl[\E^{\int_0^T c_m(X_s, U_s)\, \D{s}}\Bigr]\,,
\end{equation*}
and $\lamstr(\cG_m)\to\lamstr(\cG)$ as $m\to\infty$.
\end{lemma}

We are now ready to prove \cref{T2.7}.

\begin{proof}[Proof of \cref{T2.7}]
Let
\begin{equation*}
\lambda''(\cG)\,=\,\inf\,\Bigl\{\lambda\,\colon \exists \psi\in\Cc^2(\Rd),\
\inf_{\Rd}\psi>0, \text{\ satisfying\ } \cG\psi-\lambda\psi\le 0
\text{\ a.e.\ in}\ \Rd\Bigr\}\,.
\end{equation*}
It then follows from \cref{E-lamstr} that $\lamstr(\cG)\le \lambda''(\cG)$.
On the other hand, note that $c\le c_m$ for all $m$ large,
where $c_m$ is the function in \cref{E-cm}.
Thus, using \cref{L3.2}, we obtain
\begin{equation*}
\cG\Psi^*_m-\lamstr(\cG_m)\Psi^*_m \,\le\, 0\quad \text{in}\ \Rd,
\quad \inf_{\Rd}\,\Psi^*_m>0\,.
\end{equation*}
Therefore, $\lambda''(\cG)\le \lamstr(\cG_m)$ for all $m$,
and letting $m\to\infty$ we obtain $\lamstr(\cG)=\lambda''(\cG)$.
This concludes the proof.
\end{proof}

Next, we present the proof of \cref{T2.8}.

\begin{proof}[Proof of \cref{T2.8}]
Since the existence of a solution is known when $\lambda=\lamstr(\cG)$,
we only consider the case $\lambda>\lamstr(\cG)$.
Recall $\lambda_n$ from Lemma~\ref{L2.1}.
Since $\lim_{n\to\infty}\lambda_n=\lamstr(\cG)$,
we have $\lambda>\lambda_n$ for all $n$.
For each $n$, let $f_n$ be a non-zero, non-negative function supported
in $B_{n+1}\setminus B_n$.
Note that the principal eigenvalue of $\cG-\lambda$, in the sense of \cref{E-lamD},
is $\lamstr(\cG)-\lambda<0$.
Therefore, by \cite[Theorem~1.9]{Quaas-08a}, there exists a unique
$\varphi_n\in\Cc^2(B_{n+1})\cap\Cc(\Bar{B}_{n+1})$ satisfying
\begin{equation}\label{PT2.8A}
\cG\varphi_n-\lambda\varphi_n \,=\, - f_n\quad \text{in}\ B_{n+1},
\quad \text{and}\quad \varphi_n=0\quad \text{on}\ \partial B_{n+1}\,.
\end{equation}
Moreover, $\varphi_n\ge 0$.
Let $v_n$ be  a measurable selector of \cref{PT2.8A}, i.e.,
\begin{equation*}
a^{ij}(x)\,\partial_{ij} \varphi_n
+ b^{i}\bigl(x,v_n(x)\bigr)\,\partial_{i}\varphi_n
+ \bigl(c(x, v_n(x))-\lambda\bigr) \varphi_n \,=\,
-f_n\quad \text{in}\ B_{n+1}\,.
\end{equation*}
Applying It\^{o}'s formula, we obtain
\begin{align*}
\varphi_n(x) &\,=\, \Exp^{v_n}_x
\Bigl[\E^{\int_0^{t\wedge\uptau_{n+1}}(c(X_s, v_n(X_s))-\lambda)\,\D{s}}
\,\varphi_n(X_{t\wedge\uptau_{n+1}})\Bigr]\\
&\mspace{200mu}
+ \Exp^{v_n}_x\biggl[\int_0^{t\wedge\uptau_{n+1}}
\E^{\int_0^s (c(X_r, v_n(X_r))-\lambda)
\,\D{r}}\,f_n(X_s)\, \D{s}\biggr]
\end{align*}
for all $t\ge 0$ and $x\in B_{n+1}$.
Since $f_n\gneqq 0$, this in particular, implies that $\varphi_n>0$ in $B_{n+1}$.
We normalize $\varphi_n(0)=1$ by scaling $f_n$,
and applying Harnack's inequality to \cref{PT2.8A},
we deduce that for any compact set $K$ we can find a constant $C_K$ such that
\begin{equation*}
\norm{\varphi_n}_{\Sob^{2, p}(K)}\,<\, C_K\quad \text{for all\ } n
\text{\ sufficiently large and\ } p\in (1, \infty)\,.
\end{equation*}
It is then standard to find a $\Psi\in\Sobl^{2,p}(\Rd)$, $p\ge 1$,
such that $\varphi_n\to \Psi$ weakly in $\Sobl^{2, d}(\Rd)$ and
strongly in $\Cc^{1, \alpha}_{\mathrm{loc}}(\Rd)$ for some $\alpha\in (0,1)$.
Therefore, we can pass to the limit in \cref{PT2.8A} to obtain
\begin{equation*}
\cG\Psi \,=\, \lambda\Psi\quad \text{in\ } \Rd, \quad \text{and\ } \Psi>0\,.
\end{equation*}
Using standard regularity theory from elliptic PDE
we assert that $\Psi\in\Cc^2(\Rd)$.
This completes the proof.
\end{proof}

We conclude this section with the proof of \cref{T2.19}.
\begin{proof}[Proof of \cref{T2.19}]
Let $\Hat{v}$ be a measurable selector from
the minimizer of $\cG\varphi-\lambda\varphi=0$.
Since $\lambda\in (\lamstr(\cG), \lamstr(\cH))$,
\cref{L3.1} asserts that $\varphi$ has the stochastic representation
in \cref{EL3.1B} with $v=\Hat{v}$. Indeed, if $c$ is bounded we have
\begin{align*}
\limsup_{\abs{x}\to \infty}\,\Bigl(\max_{u\in\Act}c(x, u)-\lambda\Bigr)
&\,\le\, \limsup_{\abs{x}\to \infty}\, \max_{u\in\Act}\,c(x, u) -\lamstr(\cG)\\
&\,\le\, \limsup_{\abs{x}\to \infty}\,\max_{u\in\Act}\,c(x, u)
+\norm{c^-}_\infty\,<\,\gamma\,.
\end{align*}
In turn, the proof of \cref{T2.5} shows that either $\varphi<0$ or $\varphi=0$.
But the first option implies that $\cH(-\varphi)-\lambda(-\varphi)=0$
which contradicts the definition
of $\lamstr(\cH)$ in \cref{E-lamstr}. Hence $\varphi=0$.
\end{proof}

\section*{Acknowledgements}
The research of Ari Arapostathis was supported
in part by the National Science Foundation through grant DMS-1715210,
in part by the Army Research Office through grant W911NF-17-1-001,
and in part by the Office of Naval Research through grant N00014-16-1-2956.
The research of Anup Biswas was supported in part by an INSPIRE faculty fellowship
and DST-SERB grant EMR/2016/004810.


\end{document}